\documentclass[12pt]{amsart}
\usepackage{amsmath,amssymb,amsthm,amsfonts}
\usepackage{mathtools}
\usepackage{enumitem}

\usepackage[left=2.5cm,right=2.5cm,top=2.5cm,bottom=2.5cm]{geometry}
\usepackage[backref,colorlinks=true,linkcolor=blue,citecolor=blue,urlcolor=blue]{hyperref}
\usepackage{tikz}
\tikzset{
    bd/.style={circle, fill=black, minimum size=7pt, inner sep=0pt},
    sd/.style={circle, draw, minimum size=6pt, inner sep=0pt},
    inline/.style={baseline={([yshift=-0.5ex]current bounding box.center)}}
}

\numberwithin{equation}{section}
\numberwithin{figure}{section}

\newtheorem{theorem}{Theorem}[section]
\newtheorem{lemma}[theorem]{Lemma}
\newtheorem{proposition}[theorem]{Proposition}

\newtheorem{problem}[theorem]{Problem}
\theoremstyle{definition}

\newtheorem{remark}[theorem]{Remark}

\DeclareMathOperator{\tr}{tr}
\DeclareMathOperator{\dist}{dist}
\newcommand{\C}{\mathbb{C}}
\newcommand{\Z}{\mathbb{Z}}
\newcommand{\cP}{\mathcal{P}}

\begin{document}

\title[Degree-similar Unicyclic Graphs]{Degree-similar Unicyclic Graphs are Isomorphic}

\author[Y.-L. Zhang]{Yi-Liu Zhang}
\address{School of Mathematical Sciences, Anhui University, Hefei 230601, P. R. China}
\email{zhangyl@stu.ahu.edu.cn}

\author[Y.-Z. Fan]{Yi-Zheng Fan*}
\address{Center for Pure Mathematics, School of Mathematical Sciences, Anhui University, Hefei 230601, P. R. China}
\email{fanyz@ahu.edu.cn}
\thanks{*Corresponding author. Supported by National Natural Science Foundation of China (Grant No. 12331012).}

\subjclass[2020]{05C50}
\keywords{Degree-similar graph; unicyclic graph; graph isomorphism; adjacency matrix; degree matrix}

\begin{abstract}
Two graphs are degree-similar if their adjacency matrices and degree matrices are simultaneously similar.
Godsil and Sun asked whether non-isomorphic degree-similar unicyclic graphs exist.
We prove that they do not exist. Every graph degree-similar to a unicyclic graph is isomorphic to it.
The proof uses a symbolic leaf-peeling algorithm to recover the rooted trees attached to the unique cycle, and a rigidity lemma for colored cycles to identify automorphisms of the cycle.
\end{abstract}

\maketitle

\section{Introduction}

Throughout this paper, all graphs are finite and simple, and all matrices are over the complex field $\C$.
For a graph $G$, let $A(G)$ and $D(G)$ denote its adjacency matrix and degree matrix, respectively.
Following Godsil and Sun \cite{GodsilSun2025}, two graphs $G_1$ and $G_2$ are called \emph{degree-similar} if there exists an invertible matrix $M$ such that
\begin{equation}\label{eq:degree-similar-intro}
    A(G_1)M=MA(G_2), \quad D(G_1)M=MD(G_2).
\end{equation}
Thus, degree-similarity is substantially stronger than ordinary adjacency cospectrality: it implies similarity of the adjacency matrix, the
Laplacian, the signless Laplacian, and, in the absence of isolated vertices, the normalized Laplacian.

Degree-similarity belongs to a broad line of research on graphs that cannot be distinguished by spectra.
Classical constructions, such as Godsil-McKay switching \cite{GodsilMcKay1982}, produce non-isomorphic graphs with the same adjacency spectrum, while the generalized spectral problem also incorporates the spectrum of the complement; see the survey of van Dam and
Haemers \cite{VanDamHaemers2003}.
For more results on cospectral graphs and spectral determination, one can refer to 
\cite{HaeSpe04,JohnNew80,vanDamHaeKoo07,Wang13,Wang17,WangXuEJC06,WangXuLAA06}.
Another refinement is obtained from the generalized $\mu$-adjacency matrix
\[
    L_\mu(G):=A(G)-\mu D(G)
\]
and its $\mu$-polynomial
\[
    \psi(G,t,\mu):=\det(tI-L_\mu(G)),
\]
introduced and studied by Wang, Li, Lu, and Xu \cite{WangLiLuXu2011}.
The equality of $\mu$-polynomials simultaneously controls several familiar graph spectra and also determines the degree sequence.
Rooted coalescence provides another useful source of graphs with identical generalized polynomials \cite{ButlerEtAl2022}.
Degree-similarity places these phenomena in a more rigid matrix-theoretic setting, since \eqref{eq:degree-similar-intro} makes $L_\mu(G_1)$ and $L_\mu(G_2)$ similar through a matrix independent of $\mu$.

Godsil and Sun \cite{GodsilSun2025} initiated the systematic study of degree-similarity.
They developed several constructions of non-isomorphic degree-similar graphs and showed that equality of the $\mu$-polynomial alone does not imply degree-similarity, answering negatively an earlier question of Wang et al. \cite{WangLiLuXu2011}.
They concluded their paper with three problems that have guided the subsequent development of the subject.
We next summarize the progress on these problems.

Problem 10.1 of \cite{GodsilSun2025} asks for further constructions and, in particular, whether there exist non-isomorphic degree-similar unicyclic graphs, where a graph is called \emph{unicyclic} if it is connected and contains exactly one cycle.
Trees form the natural starting point: Godsil and Sun proved, using McKay's reconstruction theorem for trees \cite{McKaySpectralTrees}, that two degree-similar trees are necessarily isomorphic.
Hence, unicyclic graphs are the first class beyond trees in which one might expect non-isomorphic examples.
Fan, Xing, Zhang, and Wang \cite{FanXingZhangWang2026} characterized degree-similarity in terms of degree partitions, obtained new constructions, and proved that several families of unicyclic graphs are determined by degree-similarity.
They also found no non-isomorphic degree-similar unicyclic graph pairs with at most $20$ vertices.

Problem 10.2 of \cite{GodsilSun2025} concerns a finer algebraic invariant.
Degree-similar graphs have the same Smith normal form of $tI-L_\mu(G)$ over $\mathbb{Q}(\mu)[t]$, and Godsil and Sun asked whether the converse holds.
Fan, Xing, Zhang, and Wang \cite{FanXingZhangWang2026} answered this question negatively by constructing an infinite family of pairs of non-isomorphic trees with the same Smith normal form.
This phenomenon was subsequently placed in a general framework by Fan, Zhang, and Wang \cite{FanZhangWang2026SNF}: coalescing an arbitrary rooted graph at two vertices that are cospectral for $L_\mu$ preserves the complete Smith normal form.
This result connects degree-similarity to the classical use of cospectral vertices in constructing cospectral graphs.

Problem 10.3 of \cite{GodsilSun2025} asks whether deleting any two edges from a strongly regular graph always produces degree-similar graphs.
Godsil, Sun, and Zhang first proved that the resulting graphs are cospectral with respect to the adjacency, Laplacian, signless Laplacian, and normalized Laplacian matrices \cite{GodsilSunZhang2025}.
Fan, Xing, Zhang, and Wang then showed that they have the same $\mu$-polynomial \cite{FanXingZhangWang2026}.
Most recently, Fan, Wang, and Zhang \cite{FanWangZhang2026SRG} proved the stronger conclusion that edge-deleted strongly regular graphs are orthogonally degree-similar; their theorem in fact holds for all $1$-walk-regular graphs.
Thus, Problem 10.2 has a negative answer and Problem 10.3 has a positive answer, while the unicyclic graph part of Problem 10.1 of \cite{GodsilSun2025} remained open.

The purpose of this paper is to settle that remaining question.

\begin{problem}[Godsil-Sun \cite{GodsilSun2025}, Problem 10.1]
    Find more degree-similar graphs. In particular, are there non-isomorphic degree-similar unicyclic graphs?
\end{problem}

Our main result gives a complete negative answer to the second part of this problem.

\begin{theorem}\label{thm:main}
Let $G_1$ and $G_2$ be degree-similar graphs.
If $G_1$ is a unicyclic graph, then $G_1$ is isomorphic to $G_2$.
\end{theorem}

The proof is constructive.
We first express the leaf-peeling process through non-commutative polynomials in $A$ and $D$, motivated by McKay's reconstruction of trees \cite{McKaySpectralTrees}.
Simultaneous similarity then ensures that exactly the same symbolic reconstruction can be performed on two degree-similar unicyclic graphs.
The process removes the trees attached to the unique cycle while recording the rooted-tree type at each vertex of the cycle.
The remaining problem is a rigidity statement for vertex-colored cycles: the simultaneous similarity relation forces the two colorings of the cycle to differ only by a rotation or a reflection.
Matching the recorded rooted trees then yields an isomorphism of the original graphs.

The paper is organized as follows.
In Section~\ref{sec:peeling}, we give the symbolic leaf-peeling reconstruction algorithm and prove its correctness.
In Section~\ref{sec:degree-similarity}, we first use degree-similarity to show that the same symbolic reconstruction can be used for both unicyclic graphs, and reduce the problem to colored cycles; secondly, we
prove a rigidity lemma for colored cycles; finally, we complete the proof of Theorem~\ref{thm:main} by the reduction and the rigidity lemma.

\section{Leaf-peeling reconstruction for unicyclic graphs}\label{sec:peeling}
For a graph $G$ and a vertex $v$ of $G$, we write $V(G)$ for the vertex set of $G$, and $N_G(v)$ for the neighborhood of $v$.
The \emph{degree} of $v$, denoted by $\deg_G(v)$, is exactly the cardinality of $N_G(v)$.
For $U\subseteq V(G)$, let $\chi(U)$ be the diagonal characteristic matrix of $U$, that is,
 the $(v,v)$-diagonal entry of $\chi(U)$ is $1$ if $v\in U$ and is $0$ otherwise.
We use $O$ to denote a zero square matrix whose size can be known from the context.

Let $\cP=\C\langle x,y\rangle$ be the algebra of non-commutative polynomials in two variables.  If $p\in\cP$ and $G$ is a graph, write
\[
p_G=p(A(G),D(G)).
\]
We use the following symbolic version of the usual leaf-peeling procedure.
The advantage of this formulation is that every matrix that occurs in the algorithm is obtained by evaluating the same non-commutative polynomial in $A(G)$ and $D(G)$.

Let $G=(V,E)$ be a unicyclic graph with $n$ vertices and a unique cycle $C$.
The vertices outside $C$ form rooted trees attached to the vertices of $C$.
More precisely, there are rooted trees $T_1,\dots,T_t$ and a partition $\{U_i\}_{i=1}^t$ of $V(C)$ such that $G$ is obtained from $C$ by identifying, for each $i$ and each $v\in U_i$, the root of one copy of $T_i$ with $v$.
Note that some rooted tree may be trivial, that is, the tree with only one vertex.

For $0\le k\le n$, put
\[
\varepsilon_k(z)=\prod_{0\le j\le n, j\ne k}\frac{z-j}{k-j}.
\]
If $X$ is a diagonal matrix whose diagonal entries lie in $\{0,1,\dots,n\}$, then $\varepsilon_k(X)$ is the characteristic matrix of the set of diagonal positions at which $X$ has entry $k$.

Given rooted trees $R$ and $S$ and an integer $k\ge0$, let $R\Join_k S$ denote the rooted tree obtained from $R$ by adding, for each of the $k$ disjoint copies of $S$, one edge from the root of $R$ to the root of that copy of $S$, and by preserving the root of $R$ as the root of the new tree.
Thus $R\Join_0 S=R$.

\begin{lemma}\label{lem:peeling-matrix}
Let $G$ be a graph with adjacency matrix $A$ and degree matrix $D$.
Let $L$ be an independent set of degree-one vertices of $G$, and put $H=\chi(L)$ and $K=AHA$.  Then $K$ is diagonal and
\[
K_{vv}=|N_G(v)\cap L|,  \text{~for~each~vertex~} v\in V(G).
\]
Moreover, $(I-H)A(I-H)$ is the adjacency matrix of the graph obtained from $G$ by deleting all edges incident with the vertices in $L$, and $D-H-K$ is its degree matrix.
\end{lemma}

\begin{proof}
The $(u,v)$-entry of $AHA$ is the number of vertices of $L$ adjacent to both $u$ and $v$.
Since every vertex of $L$ has degree one, no vertex of $L$ can be adjacent to two distinct vertices.
Hence, all off-diagonal entries of $K$ are zero, while the diagonal entry $K_{vv}$ is exactly $|N_G(v)\cap L|$.

The matrix $(I-H)A(I-H)$ is obtained from $A$ by setting to zero all rows and columns indexed by $L$, so it deletes precisely the edges incident with $L$.
A vertex in $L$ loses its unique incident edge, accounting for the term $H$, and a vertex outside $L$ loses exactly $|N_G(v)\cap L|$ incident edges, accounting for the term $K$.
Thus, the new degree matrix is $D-H-K$.
\end{proof}

\subsection{The symbolic reconstruction algorithm}
Fix a unicyclic graph $G$ with $n$ vertices.
The algorithm works in $\cP$, but at each step zero and nonzero are tested after evaluating a certain polynomial of $\cP$ at $(A(G),D(G))$.
It maintains polynomials $a,\lambda\in\cP$ and a finite family $\Sigma$ of pairs $(S,\pi)$, where $S$ is a rooted tree and $\pi\in\cP$.

Initialize
\[
a=x,
\quad
\lambda=y,
\quad
\Sigma=\{(K_1,1)\},
\]
where $K_1$ is the trivial rooted tree with one vertex.

While $(\lambda(\lambda-2))_G\ne  O$, do the following.

\begin{enumerate}[label=\textup{(\arabic*)},leftmargin=2.2em]
\item Choose any pair $(S,\pi)\in\Sigma$ such that, for
\[
\eta=\pi\varepsilon_1(\lambda),
\]
one has $\eta_G\ne O$.
Put
\[
\xi=a\eta a.
\]

\item Replace $\Sigma$ by the family of all pairs
\[
(R\Join_k S,\,\sigma\varepsilon_k(\xi)),
\]
for each $(R,\sigma)\in\Sigma$ and each $k \in \{0,1,\ldots, n\}$ such that
$(\sigma\varepsilon_k(\xi))_G\ne O$.

\item Replace
\[
a\leftarrow (1-\eta)a(1-\eta),
\quad
\lambda\leftarrow \lambda-\eta-\xi.
\]
\end{enumerate}

After the while loop terminates, put $q=\lambda/2$ and replace $\Sigma$ by the family of all nonzero pairs
\[
(S,\pi q),
\]
for each $(S,\pi)\in\Sigma$ such that $(\pi q)_G\ne O$.
The algorithm returns the polynomial $q$ and this final family.

\begin{remark}\label{rem:symbolic-algorithm}
This is the same leaf-removal mechanism used in McKay's reconstruction algorithm for trees \cite{McKaySpectralTrees}, written in non-commuting variables.
For a fixed graph $G$, the evaluations $a_G$, $\lambda_G$, and $\pi_G$ will be the current adjacency matrix, the current degree matrix, and the characteristic matrices of the current classes.
The use of non-commutative polynomials will be essential in Section~\ref{sec:degree-similarity}.
Note also that no canonical ordering of $\Sigma$ is needed.
For comparing two degree-similar graphs, it is enough that the same symbolic choices are legal for both graphs.
\end{remark}

\begin{proposition}\label{prop:algorithm-correct}
Let $G$ be a unicyclic graph with a unique cycle $C$.
The symbolic reconstruction algorithm terminates.
If its output is
\[
q,
\quad
\{(T_i,\rho_i)\}_{i=1}^t,
\]
then
\[
q_G=\chi(V(C)),
\quad
(\rho_i)_G=\chi(U_i) ~ (1\le i\le t)
\]
for a partition $\{U_i\}_{i=1}^t$ of $V(C)$, and $G$ is obtained from $C$ by attaching one copy of $T_i$ at every vertex in $U_i$.
\end{proposition}

\begin{proof}
At each stage write
\[
A_{\mathrm{work}}=a_G,
\quad
D_{\mathrm{work}}=\lambda_G,
\quad
\Sigma_G=\{(S,\pi_G):(S,\pi)\in\Sigma\}.
\]
We prove that these evaluated data satisfy the usual leaf-peeling invariants.

Let $G_{\mathrm{work}}^{(q)}$ be the graph represented by $A_{\mathrm{work}}$ and $D_{\mathrm{work}}$ after $q$ iterations, and put
\[
W_q=\{v\in V(G):\deg_{G_{\mathrm{work}}^{(q)}}(v)>0\}.
\]
We prove by induction on $q$ the following assertions.

\begin{enumerate}[label=\textup{(A\arabic*)},leftmargin=2.6em]
\item The non-isolated part $G_{\mathrm{work}}^{(q)}[W_q]$ is a connected unicyclic graph containing the original cycle $C$.

\item If $\Sigma^{(q)}=\{(S_\mu,\pi_\mu)\}_\mu$, then $(\pi_\mu)_G=\chi(U_\mu^{(q)})$ for some sets $U_\mu^{(q)}$, and these sets form a partition of $V(G)$.

\item If $u\in W_q\cap U_\mu^{(q)}$, then the connected component containing the vertex $u$ in the graph with edge set $E(G)\setminus E(G_{\mathrm{work}}^{(q)})$ is a rooted tree isomorphic to $S_\mu$, rooted at $u$.
\end{enumerate}

For $q=0$ the assertions are immediate, since $a_G=A(G)$, $\lambda_G=D(G)$, and the only member of $\Sigma$ is $(K_1,1)$.

Assume the assertions hold at the beginning of an iteration.
If $(\lambda(\lambda-2))_G\ne O$, then the non-isolated part is not a cycle, and hence it has a degree-one vertex.
By (A2), some member of the current partition contains such a vertex, so a pair $(S,\pi)$ with $(\pi\varepsilon_1(\lambda))_G\ne O$ exists.
The chosen pair only fixes one possible legal peeling order; the following argument works for every legal choice.

Let $\eta=\pi\varepsilon_1(\lambda)$ and put $H=\eta_G$.  If $\pi_G=\chi(U_S^{(q)})$, then
\[
H=\chi(L),
\quad
L=\{v\in U_S^{(q)}:\deg_{G_{\mathrm{work}}^{(q)}}(v)=1\}.
\]
The set $L$ is independent: no vertex of $C$ has degree one, and two adjacent degree-one vertices would form a connected component consisting of a single edge, contradicting that the non-isolated part is connected and contains $C$.
Applying Lemma~\ref{lem:peeling-matrix}, with $\xi=a\eta a$, the matrix $\xi_G=A_{\mathrm{work}}HA_{\mathrm{work}}$ is diagonal and its $(v,v)$-diagonal entry is $|N_{G_{\mathrm{work}}^{(q)}}(v)\cap L|$.
The update
\[
a\leftarrow (1-\eta)a(1-\eta),
\quad
\lambda\leftarrow \lambda-\eta-\xi
\]
therefore evaluates to the adjacency matrix and degree matrix obtained by isolating the vertices in $L$.
The new non-isolated part remains connected unicyclic and still contains $C$, proving (A1).

For (A2), fix an old entry $(R,\sigma)$ with $\sigma_G=\chi(U_R^{(q)})$.
Since the diagonal entries of $\xi_G$ count the number of neighbours in $L$,
\[
(\sigma\varepsilon_k(\xi))_G
=
\chi\bigl(U_R^{(q)}\cap\{v:|N_{G_{\mathrm{work}}^{(q)}}(v)\cap L|=k\}\bigr).
\]
As $R$ ranges over the old entries and $k$ ranges over $0,\ldots,n$, the nonzero sets above form a refinement of the old partition.
Thus the evaluated new family again consists of characteristic matrices of a partition.

For (A3), let $u$ be a non-isolated vertex after the update and suppose that $u$ belongs to the new part indexed by $R\Join_k S$.  Before the update, $u$ carried the rooted tree $R$, and it has exactly $k$ neighbours in $L$.  Each of these neighbours carried the rooted tree $S$.  Hence the peeled component rooted at $u$ is obtained from $R$ by attaching $k$ copies of $S$ at the root, namely it is $R\Join_k S$.

Each iteration removes at least one vertex from the non-isolated part, so the algorithm terminates.
At termination every vertex of the non-isolated part has degree two.
Since this part is connected unicyclic and contains $C$, it is exactly $C$.
Thus $(\lambda/2)_G=\chi(V(C))$.
The final multiplication by $q=\lambda/2$ restricts the partition to the cycle vertices, and (A3) shows that the rooted tree recorded at a cycle vertex is precisely the rooted tree attached to it in $G-E(C)$.
This proves the proposition.
\end{proof}

\subsection{An example for the algorithm}
We give a small example to show how the symbolic algorithm should be read.
The point of the example is not to compute many polynomials, but to see what the two matrices $\eta$ and $\xi$ mean during the leaf-peeling process.

Let $G$ be unicyclic a graph listed in Fig. \ref{exam}, which has the vertex set $V=\{1,\dots,13\}$ and the edge set
\[
\begin{split}
E=\{&\{1,2\},\{2,3\},\{3,4\},\{4,5\},\{5,1\},\{1,6\},\{6,7\},\{6,8\},\{3,9\},\{3,10\},\\
&\{5,11\},\{11,12\},\{12,13\}\}.
\end{split}
\]

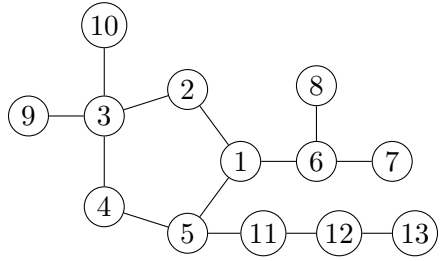
\begin{figure}[htbp]
\begin{center}
\begin{tikzpicture}[every node/.style={circle,draw,inner sep=1.4pt,minimum size=15pt,font=\small}]
\node (1) at (1,0) {1};
\node (2) at (0.3,0.95) {2};
\node (3) at (-0.8,0.6) {3};
\node (4) at (-0.8,-0.6) {4};
\node (5) at (0.3,-0.95) {5};
\node (6) at (2,0) {6};
\node (7) at (3,0) {7};
\node (8) at (2,1) {8};
\node (9) at (-1.8,0.6) {9};
\node (10) at (-0.8,1.8) {10};
\node (11) at (1.3,-0.95) {11};
\node (12) at (2.3,-0.95) {12};
\node (13) at (3.3,-0.95) {13};
\draw (1)--(2)--(3)--(4)--(5)--(1);
\draw (1)--(6)--(7); \draw (6)--(8);
\draw (3)--(9); \draw (3)--(10);
\draw (5)--(11)--(12)--(13);
\end{tikzpicture}
\end{center}
\caption{\small A unicyclic graph $G$ on $13$ vertices}\label{exam}
\end{figure}

The unique cycle $C$ has vertices $1,2,3,4,5$.
We will use $X^{(q)}$ to indicate the variable $X$ at the $q$-th step.

\textbf{Initialization.}
Put $a^{(0)}=x, \lambda^{(0)}=y, \Sigma^{(0)}=\{(K_1,1)\}$.
 It gives $a^{(0)}_G=A(G)$, $\lambda^{(0)}_G=D(G)$, $\Sigma^{(0)}_G=\{(K_1,\chi(V))\}$, and the working graph $G^{(0)}=G$.
At every step, one may imagine that each vertex carries a rooted tree already peeled off below it.
At first, every vertex carries only $K_1$ and the working graph $G^{(0)}=G$.
If the selected entry is $(S,\pi)$, then $\eta=\pi\varepsilon_1(\lambda)$ marks the leaves of the current working graph that also lie in the selected class determined by $\pi$ (recalling that $\pi_G=\chi(U)$ for some subset $U \subset V(G)$), while $\xi=a\eta a$ marks their neighbors and counts how many selected leaves are attached to each neighbor.
A vertex that previously carried $R$ and is hit $k$ times by $\xi$ now carries $R\Join_k S$.

\textbf{Step 1.}
The selected class is $(K_1,1)$.
Thus, the current leaves are $L^{(1)}=\{7,8,9,10,13\}$ and $\eta^{(1)}_G=\chi(L^{(1)})$.
Their neighbours are counted by $\xi^{(1)}_G=2\chi(\{3,6\})+\chi(\{12\})$, so vertices $3$ and $6$ receive two copies of $K_1$, while vertex $12$ receives one copy of $K_1$.
Let
$$ T_1=K_1\Join_1 K_1=\begin{tikzpicture}[inline]
    \node[bd] (A) at (0,0) {};
    \node[sd] (B) at (0.8,0) {};
    \draw (A) -- (B);
\end{tikzpicture}, \quad
T_2=K_1\Join_2 K_1=\begin{tikzpicture}[inline]
    \node[bd] (A) at (0,0) {};
    \node[sd] (B) at (0.8,0) {};
    \node[sd] (C) at (0,0.8) {};
    \draw (C) -- (A) -- (B);
\end{tikzpicture},
$$
where the dark vertices denote roots.
Then $T_1$ is carried at the vertex $12$, and $T_2$ is  carried at the vertices $3$ and $6$, and all other vertices still carry $K_1$.
The working graph $G^{(1)}$ is listed in Fig. \ref{step1}.
The family
$$ \Sigma_G^{(1)}=\{(K_1, \chi(V\setminus\{3,6,12\})), (T_1, \chi(\{12\})), (T_2, \chi(\{3,6\}))\}.$$
Equivalently, the first refinement has the three nonzero parts
$K_1$ on $V\setminus\{3,6,12\}$, $T_1$ on $\{12\}$, and $T_2$ on $\{3,6\}$.

\begin{figure}[htbp]
\begin{center}
\begin{tikzpicture}[every node/.style={circle,draw,inner sep=1.4pt,minimum size=15pt,font=\small}]
\node (1) at (1,0) {1};
\node (2) at (0.3,0.95) {2};
\node (3) at (-0.8,0.6) {3};
\node (4) at (-0.8,-0.6) {4};
\node (5) at (0.3,-0.95) {5};
\node (6) at (2,0) {6};
\node (7) at (3,0) {7};
\node (8) at (2,1) {8};
\node (9) at (-1.8,0.6) {9};
\node (10) at (-0.8,1.8) {10};
\node (11) at (1.3,-0.95) {11};
\node (12) at (2.3,-0.95) {12};
\node (13) at (3.3,-0.95) {13};
\draw (1)--(2)--(3)--(4)--(5)--(1);
\draw (1)--(6);
\draw (5)--(11)--(12);
\end{tikzpicture}
\end{center}
\caption{\small  The working graph $G^{(1)}$ after Step 1}\label{step1}
\end{figure}

\textbf{Step 2.}
Now choose the pair of $\Sigma^{(1)}$ carrying $T_1$.
In the current graph, the only leaf in this class is the vertex $12$.
Thus, $\eta^{(2)}_G=\chi(\{12\})$, and
$\xi^{(2)}_G=\chi(\{11\})$.
The vertex $11$, which previously carried $K_1$, now carries the rooted tree
$$T_3=K_1\Join_1 T_1=
\begin{tikzpicture}[inline]
    \node[bd] (A) at (0,0) {};
    \node[sd] (B) at (0.8,0) {};
    \node[sd] (C) at (1.6,0) {};
    \draw (A) -- (B) -- (C);
\end{tikzpicture}.$$
The working graph $G^{(2)}$ is listed in Fig. \ref{step2}.
The family
$$ \Sigma_G^{(2)}=\{(K_1, \chi(V\setminus\{3,6,11,12\})), (T_1, \chi(\{12\})), (T_2, \chi(\{3,6\})), (T_3, \chi(\{11\})\}.$$

\begin{figure}[htbp]
\begin{center}
\begin{tikzpicture}[every node/.style={circle,draw,inner sep=1.4pt,minimum size=15pt,font=\small}]
\node (1) at (1,0) {1};
\node (2) at (0.3,0.95) {2};
\node (3) at (-0.8,0.6) {3};
\node (4) at (-0.8,-0.6) {4};
\node (5) at (0.3,-0.95) {5};
\node (6) at (2,0) {6};
\node (7) at (3,0) {7};
\node (8) at (2,1) {8};
\node (9) at (-1.8,0.6) {9};
\node (10) at (-0.8,1.8) {10};
\node (11) at (1.3,-0.95) {11};
\node (12) at (2.3,-0.95) {12};
\node (13) at (3.3,-0.95) {13};
\draw (1)--(2)--(3)--(4)--(5)--(1);
\draw (1)--(6);
\draw (5)--(11);
\end{tikzpicture}
\end{center}
\caption{\small The working graph $G^{(2)}$ after Step 2}\label{step2}
\end{figure}

\textbf{Step 3.}
Choose the pair of $\Sigma^{(2)}$ carrying $T_2$.
This class contains vertices $3$ and $6$, but only $6$ is a leaf of the current working graph.
Hence $\eta^{(3)}_G=\chi(\{6\})$ and
$\xi^{(3)}_G=\chi(\{1\})$.
The vertex $1$ previously carried $K_1$, and now carries the rooted tree
$$T_4=K_1\Join_1T_2=
\begin{tikzpicture}[inline]
    \node[sd] (A) at (0,0) {};
    \node[sd] (B) at (0.8,0) {};
    \node[sd] (C) at (0,0.8) {};
    \node[bd] (D) at (-0.8,0) {};
    \draw (C) -- (A) -- (B); \draw (D) -- (A);
\end{tikzpicture}.$$
The working graph $G^{(3)}$ is listed in Fig. \ref{step3}.
The family
$$ \Sigma_G^{(3)}=\{(K_1, \chi(V\setminus\{1,3,6,11,12\})), (T_1, \chi(\{12\})), (T_2, \chi(\{3,6\})), (T_3, \chi(\{11\}), (T_4,\chi(\{1\}))\}.$$

\begin{figure}[htbp]
\begin{center}
\begin{tikzpicture}[every node/.style={circle,draw,inner sep=1.4pt,minimum size=15pt,font=\small}]
\node (1) at (1,0) {1};
\node (2) at (0.3,0.95) {2};
\node (3) at (-0.8,0.6) {3};
\node (4) at (-0.8,-0.6) {4};
\node (5) at (0.3,-0.95) {5};
\node (6) at (2,0) {6};
\node (7) at (3,0) {7};
\node (8) at (2,1) {8};
\node (9) at (-1.8,0.6) {9};
\node (10) at (-0.8,1.8) {10};
\node (11) at (1.3,-0.95) {11};
\node (12) at (2.3,-0.95) {12};
\node (13) at (3.3,-0.95) {13};
\draw (1)--(2)--(3)--(4)--(5)--(1);
\draw (5)--(11);
\end{tikzpicture}
\end{center}
\caption{\small The working graph $G^{(3)}$ after Step 3}\label{step3}
\end{figure}

\textbf{Step 4.}
Choose the pair of $\Sigma^{(3)}$ carrying $T_3$.
The only current leaf in this class is vertex $11$.
Thus $\eta^{(4)}_G=\chi(\{11\})$, and $\xi^{(4)}_G=\chi(\{5\})$.
The vertex $5$ previously carried $K_1$, so it now carries the rooted tree
$$T_5=K_1 \Join_1 T_3=
\begin{tikzpicture}[inline]
    \node[bd] (A) at (0,0) {};
    \node[sd] (B) at (0.8,0) {};
    \node[sd] (C) at (1.6, 0) {};
    \node[sd] (D) at (2.4,0) {};
    \draw (A) -- (B) -- (C) -- (D);
\end{tikzpicture}.$$
The working graph $G^{(4)}$ is now just the cycle with several isolated vertices; see Fig. \ref{step4}.
The family
\[ \begin{split}
\Sigma_G^{(4)}= \{& (K_1, \chi(V\setminus\{1,3,5, 6,11,12\})), (T_1, \chi(\{12\})), (T_2, \chi(\{3,6\})), (T_3, \chi(\{11\})), \\
&(T_4,\chi(\{1\})), (T_5, \chi(\{5\}))\}.
\end{split}
\]

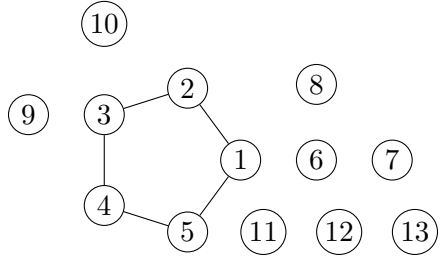
\begin{figure}[htbp]
\begin{center}
\begin{tikzpicture}[every node/.style={circle,draw,inner sep=1.4pt,minimum size=15pt,font=\small}]
\node (1) at (1,0) {1};
\node (2) at (0.3,0.95) {2};
\node (3) at (-0.8,0.6) {3};
\node (4) at (-0.8,-0.6) {4};
\node (5) at (0.3,-0.95) {5};
\node (6) at (2,0) {6};
\node (7) at (3,0) {7};
\node (8) at (2,1) {8};
\node (9) at (-1.8,0.6) {9};
\node (10) at (-0.8,1.8) {10};
\node (11) at (1.3,-0.95) {11};
\node (12) at (2.3,-0.95) {12};
\node (13) at (3.3,-0.95) {13};
\draw (1)--(2)--(3)--(4)--(5)--(1);
\end{tikzpicture}
\end{center}
\caption{\small The working graph $G^{(4)}$ after Step 4}\label{step4}
\end{figure}

At this point every non-isolated vertex has current degree two, so the loop stops.
Put $q=\lambda^{(4)}/2$. Then
\[
q_G=\chi(\{1,2,3,4,5\})=\chi(V(C)).
\]
Multiplying the final indexed family $\Sigma_G^{(4)}$ by $q$ simply discards all isolated vertices.
The remaining family is
\[
(K_1,\chi(\{2,4\})),\quad (T_2,\chi(\{3\})),\quad (T_4,\chi(\{1\})),\quad (T_5,\chi(\{5\})).
\]
Thus the algorithm recovers exactly the visible decomposition of $G$: the cycle $C$ on vertices $1,\ldots,5$, together with $T_2$ attached at vertex $3$, $T_4$ attached at vertex $1$, $T_5$ attached at vertex $5$; and no nontrivial tree is attached at vertices $2$ and $4$.

\section{Degree similarity of unicyclic graphs}\label{sec:degree-similarity}

\begin{lemma}\label{lem:polynomial-intertwining}
Let $G_1$ and $G_2$ be degree-similar via an invertible matrix $M$, so that $A(G_1)M=MA(G_2)$ and $D(G_1)M=MD(G_2)$.
Then, for every $p \in\cP$,
\[
p(A(G_1),D(G_1))M=Mp(A(G_2),D(G_2)).
\]
In particular, $p(A(G_1),D(G_1))= O$ if and only if $p(A(G_2),D(G_2))= O$.
\end{lemma}

\begin{proof}
The identity is immediate for $p=x$ and $p=y$, and is preserved under addition, scalar multiplication, and multiplication of non-commutative polynomials.
The final statement follows by multiplying by $M^{-1}$.
\end{proof}

\begin{lemma}\label{lem:degree-similar-preserves-unicyclic}
If $G_1$ and $G_2$ are degree-similar and $G_1$ is a unicyclic graph, then $G_2$ is also a unicyclic graph.
\end{lemma}

\begin{proof}
By \cite[Lemma 17]{FanXingZhangWang2026}, $G_2$ is connected with the same number of edges as $G_1$, implying that $G_2$ is unicyclic.
\end{proof}

\begin{proposition}\label{prop:cycle-reduction}
Let $G_1$ and $G_2$ be degree-similar unicyclic graphs.
Then, after relabeling the vertices, their unique cycles can be identified with the same cycle $C$.
Moreover, there are finite families
\[
\Sigma^1=\{(T_i,\chi(U_i^1))\}_{i=1}^t,
\quad
\Sigma^2=\{(T_i,\chi(U_i^2))\}_{i=1}^t,
\]
where $\{U_i^\ell\}_{i=1}^t$ is a partition of $V(C)$ for $\ell=1,2$, such that $G_\ell$ is reconstructed by attaching $T_i$ to every vertex in $U_i^\ell$.
In addition, $|U_i^1|=|U_i^2|$ for all $i$, and there exists an invertible matrix $M_C$ indexed by $V(C)$ such that
\[
A(C)M_C=M_CA(C),
\quad
\chi(U_i^1)M_C=M_C\chi(U_i^2) ~ (1\le i\le t),
\]
where $\chi(U_i^\ell)$ denotes the characteristic matrix on the vertex set $V(C)$.
\end{proposition}

\begin{proof}
Let $M$ be an invertible matrix that satisfies $A_1 M=M A_2, D_1 M=M D_2$, where $A_i,D_i$ be the adjacency and degree matrices of $G_i$ for $i=1,2$.
Run any symbolic legal reconstruction of Section~\ref{sec:peeling} with $G_1$ as input.
The Lemma~\ref{lem:polynomial-intertwining} shows that every zero test made in this run has the same answer on $G_1$ and on $G_2$.
Hence, the same sequence of symbolic choices is also legal for $G_2$: the same pairs may be chosen, the same entries are omitted, and the same rooted trees and output polynomials are obtained.
Thus, the proof does not require a canonical ordering of $\Sigma$.

Let the common symbolic output be
\[
q, \quad
\{(T_i,\rho_i)\}_{i=1}^t.
\]
By Proposition~\ref{prop:algorithm-correct}, for $\ell=1,2$,
\[
q(A_\ell,D_\ell)=\chi_C^\ell,
\quad
\rho_i(A_\ell,D_\ell)=\chi(U_i^\ell) ~ (1\le i\le t),
\]
where $\chi_C^\ell$ is the characteristic matrix of the unique cycle of $G_\ell$, the sets $\{U_i^\ell\}_{i=1}^t$ partition the cycle $C$, and $G_\ell$ is reconstructed by attaching $T_i$ at every vertex in $U_i^\ell$.  Applying Lemma~\ref{lem:polynomial-intertwining} to $q$ and to each $\rho_i$ gives
\begin{equation}\label{eq:cycle-projection-intertwining}
\chi_C^1 M=M \chi_C^2,
\end{equation}
\begin{equation}\label{eq:color-projection-intertwining}
\chi(U_i^1)M=M\chi(U_i^2), ~ 1\le i\le t.
\end{equation}
By \eqref{eq:cycle-projection-intertwining}, $\chi_C^1$ and $\chi_C^2$ have the same rank, so two cycles have the same length, say $r$.
Similarly, by \eqref{eq:color-projection-intertwining}, $|U_i^1|=|U_i^2|$ for every $i$.

Relabel the vertices so that the two unique cycles have the same vertex set $V(C)=\{1,\ldots,r\}$.
  Under the corresponding permutation similarities, $M$ is replaced by the transformed similarity matrix and is still denoted by $M$, and \eqref{eq:cycle-projection-intertwining} becomes
\begin{equation}\label{M-decom}
\chi(V(C)) M= M \chi(V(C)).
\end{equation}
So, $$M=M[V(C)] \oplus M[V(G)\setminus V(C)],$$
 where $M[S]$ denotes the principal submatrix of $M$ indexed by the vertex subset $S \subset V(G)$.
Surely, $M_C:=M[V(C)]$ is invertible.

By the equality $A_1M=MA_2$ and the direct sum decomposition of $M$, we have
\[
A(C) M_C=M_C A(C).
\]
The restriction of \eqref{eq:color-projection-intertwining} to $V(C)$ gives
\[
\chi(U_i^1)M_C=M_C\chi(U_i^2)
\quad(1\le i\le t),
\]
where now $\chi(U_i^\ell)$ denotes the characteristic matrix on $V(C)$.
This proves the proposition.
\end{proof}

Next, we prove a rigidity lemma for colored cycles.
The \emph{distance} of two vertices $u,v$ in a graph $G$, denoted by $\dist_G(u,v)$, is defined to be the minimum length of the paths of $G$ connecting $u$ and $v$.

\begin{lemma}\label{lem:colored-cycle}
Let $C$ be the cycle with vertex set $V=\Z_r$ and adjacency matrix $A$.
Let $\{U_k^1\}_{k=1}^t$ and $\{U_k^2\}_{k=1}^t$ be two partitions of $V$.
Suppose that there exists an invertible matrix $M$ such that
\begin{equation}\label{simu-simi}
AM=MA,
\quad
\chi(U_k^1)M=M\chi(U_k^2)
\quad(1\le k\le t).
\end{equation}
Then there exists an automorphism $\varphi$ of $C$ such that
\[
\varphi(U_k^1)=U_k^2, ~ 1\le k\le t.
\]
\end{lemma}

\begin{remark}
For a graph, a partition of its vertex set is equivalent to a corloring of its vertices.
In Lemma \ref{lem:colored-cycle}, the partition $\{U_k^1\}_{k=1}^t$ gives a colored cycle denoted by $C^{(1)}$, that is, a vertex $v$ receives color $k$ if $v \in U_k^1$.
Similarly, the partition $\{U_k^2\}_{k=1}^t$ also gives a colored cycle $C^{(2)}$.
The lemma \ref{lem:colored-cycle} says that if the two colored cycles $C^{(1)}$ and $C^{(2)}$ satisfy the simultaneous similarity relation \eqref{simu-simi}, then the two colorings (i.e. two partitions) differ only by a rotation or a reflection.
\end{remark}

\begin{proof}
For $0\le d\le\lfloor r/2\rfloor$, let $B_d$ be the distance-$d$ matrix of $C$, that is,
$(B_d)_{uv}=1$ if and only if $\dist_C(u,v)=d$.
We first prove that every $B_d$ is a polynomial in $A$.
Indeed, let $S$ be the cyclic shift matrix indexed by $\Z_r$, that is, $S_{i,j}=1$ if and only $j=i+1 \mod r$.
Then $A=S+S^{-1}$.
If the polynomials $f_d$ are defined by $f_0(x)=2, f_1(x)=x$ and
$$ f_{d+1}(x)=xf_d(x)-f_{d-1}(x), $$
then $f_d(A)=S^d+S^{-d}$ by induction.
Thus, $B_0=I$, $B_d=S^d+S^{-d}=f_d(A)$ for $1\le d<r/2$, and $B_{r/2}=\frac{1}{2}f_{r/2}(A)$ if $r$ is even.
Therefore, $B_d M=M B_d$ for all $d$ as $A M = M A$.

Let $\omega=(\omega_0,\dots,\omega_{N-1})$ be a word over $\{1,\dots,t\}$, and let $\mathbf d=(d_1,\dots,d_N)$ be with $0\le d_i\le\lfloor r/2\rfloor$.
Define
\[
T_\ell(\omega,\mathbf d)
=
\tr\bigl(
\chi(U_{\omega_0}^\ell) B_{d_1}
\chi(U_{\omega_1}^\ell) B_{d_2}
\cdots
\chi(U_{\omega_{N-1}}^\ell) B_{d_N}
\bigr).
\]
Using $\chi(U_k^1)=M \chi(U_k^2) M^{-1}$ and $B_d=M B_d M^{-1}$, we obtain
\[
T_1(\omega,\mathbf d)=T_2(\omega,\mathbf d).
\]
Combinatorially, $T_\ell(\omega,\mathbf d)$ is the number of closed sequences $x_0,x_1,\dots,x_N=x_0$ in $C$ such that
\[
x_i\in U_{\omega_i}^\ell ~ (0\le i \le N-1),
\quad
\dist_C(x_{i-1},x_i)=d_i ~ (1\le i\le N).
\]

We now choose a closed distance word that is rigid up to an automorphism of the cycle.

First, suppose that $r$ is odd.
Set $N=r$ and $v_i=i\pmod r$ for $0\le i\le r$.
Then $v_r=v_0$, every successive distance is $1$, and the sequence visits every vertex.
If $x_0,x_1,\dots,x_r=x_0$ is any closed sequence with all successive distances equal to $1$, namely,
$$\dist_C(x_{i-1},x_i)=\dist_C(v_{i-1},v_i)=1, ~ 1 \le i \le N,$$
then $x_i-x_{i-1}=\varepsilon_i$ for some $\varepsilon_i\in\{\pm1\}$.
The closing condition gives $\sum_{i=1}^r\varepsilon_i\equiv0\pmod r$.
The integer $\sum_i\varepsilon_i$ lies between $-r$ and $r$ and has odd parity, so it is either $r$ or $-r$.
Thus, all signs are equal, and the sequence is obtained from $v_0,\dots,v_r$ by a rotation or a reflection of $C$.

Now suppose that $r=2m$ is even.
Since $r\ge4$, we have $m\ge2$.
Define a closed sequence $v_0,\dots,v_N=v_0$ by setting $v_0=0$ and using the following increments in $\Z_{2m}$:
\[
(s_1,\dots,s_N)=
\begin{cases}
(m,1,m,1,\dots,m,1), & m \text{~odd},\\
(m,1,m,1,\dots,m,1,m), & m \text{~even},
\end{cases}
\]
where in both lines the pair $(m,1)$ is repeated $m$ times.
Put $v_i=v_{i-1}+s_i$.
The number of increments equal to $m$ is odd, and the number of increments equal to $1$ is $m$; hence, the total displacement is congruent to $m+m=0 \mod {2m}$, so the sequence is closed.

We claim that the sequence visits every vertex of $C$.
If $m$ is odd, then
\[
v_{2a}=a(m+1),
~
v_{2a+1}=a(m+1)+m, ~ 0\le a \le m-1.
\]
Here $\gcd(m+1,2m)=2$, so the vertices $v_{2a}$ run through the even parity class of $V(C)$ and the vertices $v_{2a+1}$ run through the odd parity class of $V(C)$.
If $m$ is even, then $\gcd(m+1,2m)=1$.
The vertices $a(m+1)$ with $0\le a\le m$ are distinct, and the vertices $a(m+1)+m$ with $0\le a<m$ are distinct.
If $a(m+1)\equiv b(m+1)+m\pmod {2m}$ with $0\le a\le m$ and $0\le b<m$, then, since $m+1$ is invertible modulo $2m$ and $m(m+1)\equiv m\pmod {2m}$, we get $a-b\equiv m\pmod {2m}$.
The given ranges force $a=m$ and $b=0$.
Thus, the two lists intersect only at the common vertex $m$, and their union has $2m$ vertices.

Let $x_0,\dots,x_N=x_0$ be any closed sequence that
\[
\dist_C(x_{i-1},x_i)=\dist_C(v_{i-1},v_i), ~ 1\le i\le N.
\]
The steps of distance $m$ are forced, because each vertex of $C_{2m}$ has a unique antipodal vertex.
Each step of distance $1$ has the sign $+1$ or $-1$; write these signs as $\varepsilon_1,\ldots,\varepsilon_m$.
Since the number of antipodal steps is odd, their total contribution is congruent with $m$ modulo $2m$.
The closing sequence condition, therefore, gives
\[
\sum_{j=1}^m\varepsilon_j\equiv m \mod {2m}.
\]
The integer $\sum_j\varepsilon_j$ lies between $-m$ and $m$, so it is either $m$ or $-m$.
Hence, all unit steps have the same sign.
Therefore, $x_0,\dots,x_N$ is obtained from $v_0,\dots,v_N$ by a rotation or a reflection of $C$.

In both parity cases of $r$, we have constructed a closed sequence $S: v_0,\dots,v_N=v_0$ of $C$ that visits every vertex of $C$ and has the following rigidity property: every closed sequence of $C$ with the same successive distance as $S$ is obtained from $S$ by an automorphism (rotation or reflection) of $C$.

Now, define a particular word $\omega=(\omega_0,\dots,\omega_{N-1})$ by $\omega_i=k$ if $v_i\in U_k^1$,
and $\mathbf d=(d_1,\ldots, d_N)$ by $d_i=\dist_C(v_{i-1},v_i)$ for $1\le i\le N$.
The sequence $v_0,\ldots,v_N$ is counted by $T_1(\omega,\mathbf d)$, so $T_1(\omega,\mathbf d)>0$.
Since $T_1(\omega,\mathbf d)=T_2(\omega,\mathbf d)$, there exists a sequence $x_0,\ldots,x_N=x_0$ counted by $T_2(\omega,\mathbf d)$.
By the rigidity property, there is an automorphism $\varphi$ of $C$ such that $x_i=\varphi(v_i)$ for all $i$.

For every $i<N$, the condition that $x_i$ is counted by $T_2$ gives $x_i\in U_{\omega_i}^2$, while $v_i\in U_{\omega_i}^1$ by definition of $\omega_i$.  Since the sequence $v_0,\dots,v_N$ visits every vertex of $C$, we get
\[
\varphi(U_k^1)\subseteq U_k^2, ~ 1\le k\le t.
\]
The sets $\{\varphi(U_k^1)\}_{k=1}^t$ and $\{U_k^2\}_{k=1}^t$ are both partitions of $V$, so all these inclusions are equalities.  This proves the lemma.
\end{proof}

Finally, we arrive at the proof of the main result, namely, Theorem \ref{thm:main}.

\begin{proof}[Proof of Theorem \ref{thm:main}]
By Proposition~\ref{prop:cycle-reduction}, after relabelling the unique cycles of $G_1$ and $G_2$ as the same cycle $C$, there are finite families
\[
\Sigma^1=\{(T_i,\chi(U_i^1))\}_{i=1}^t,
\quad
\Sigma^2=\{(T_i,\chi(U_i^2))\}_{i=1}^t
\]
that reconstruct $G_1$ and $G_2$, respectively, and there exists an invertible matrix $M_C$ such that
\[
A_CM_C=M_CA_C,
\quad
\chi(U_i^1)M_C=M_C\chi(U_i^2)
\quad(1\le i\le t).
\]
By Lemma~\ref{lem:colored-cycle}, there is an automorphism $\varphi$ of $C$ satisfying $\varphi(U_i^1)=U_i^2$ for every $i$.

For each $i$ and each $v\in U_i^1$, let $T_i^v$ be the copy of $T_i$ attached at $v$ in $G_1$, and let $T_i^{\varphi(v)}$ be the copy of $T_i$ attached at $\varphi(v)$ in $G_2$.
Choose a root-preserving isomorphism
\[
\iota_{i,v}:T_i^v\longrightarrow T_i^{\varphi(v)}.
\]
Define a map $\psi:V(G_1)\to V(G_2)$ by setting $\psi(v)=\varphi(v)$ for $v\in V(C)$ and $\psi(x)=\iota_{i,v}(x)$ when $x$ lies in the copy $T_i^v$ attached at $v$.
This is well-defined because the root of $T_i^v$ is identified with $v$, and $\iota_{i,v}$ sends that root to $\varphi(v)$.

The map $\psi$ is bijective.
It preserves edges on the cycle because $\varphi$ is a cycle automorphism, and it preserves all edges inside the attached rooted trees because each $\iota_{i,v}$ is a graph isomorphism.
These are all edges of a unicyclic graph decomposed into its cycle and attached rooted trees.
Hence $\psi$ is a graph isomorphism from $G_1$ to $G_2$.
\end{proof}

\end{document}